\documentclass[10pt]{amsart}

\usepackage{enumerate}
\usepackage{amssymb}
\usepackage{amsmath,amscd}
\usepackage[colorinlistoftodos,disable]{todonotes}
\usepackage{amsthm}
\usepackage{tikz-cd}

\newcommand{\Har}{\mathcal{H}}

\newcommand{\Hess}{\operatorname{Hess}}

\newcommand{\F}{\mathcal{F}}
\newcommand{\B}{\mathcal{B}}
\newcommand{\Le}{\mathcal{L}}

\newcommand{\scal}[1]{\langle #1 \rangle}
\newcommand{\In}{\subset}

\newcommand{\R}{\mathbb{R}}

\newcommand{\C}{\mathbb{C}}

\newcommand{\HH}{\mathbb{H}}

\newcommand{\OO}{\operatorname{O}}
\newcommand{\SO}{\operatorname{SO}}

\newcommand{\U}{\operatorname{U}}

\newcommand{\Sp}{\operatorname{Sp}}

\newcommand{\Sym}{\operatorname{Sym}}

\newtheorem{theorem}{Theorem}

\newtheorem{lemma}[theorem]{Lemma}

\newtheorem{maintheorem}{Theorem}

\newtheorem{maincor}[maintheorem]{Corollary}

\theoremstyle{definition}
\newtheorem{definition}[theorem]{Definition}

\theoremstyle{remark}
\newtheorem{remark}[theorem]{Remark}
\newtheorem{example}[theorem]{Example}
\newtheorem{question}[theorem]{Question}

\title[Maximality of Laplacian algebras]{Maximality of Laplacian algebras, with applications to Invariant Theory}
\date{}

\author[R.~Mendes]{Ricardo A. E. Mendes}
\address{University of Oklahoma, USA}
\email{ricardo.mendes@ou.edu}

\author[M.~Radeschi]{Marco Radeschi}
\address{University of Notre Dame, USA}
\email{mradesch@nd.edu}

\thanks{The first-named author has been supported by the NSF grant DMS-2005373, and the second-named author by NSF 1810913.}

\begin{document}

\begin{abstract}
We show Laplacian algebras are maximal, and give applications to the Classical Invariant Theory of real orthogonal representations of compact groups, including: The solution of the Inverse Invariant Theory problem for finite groups. An if-and-only-if criterion for when a separating set is a generating set. And the introduction of a class of generalized polarizations which, in a certain class of representations (including all representations of finite groups), always generates the algebra of invariants of their diagonal representations.

%We show Laplacian algebras are maximal, and give two applications to the Classical Invariant Theory of real orthogonal representations. The first is an if and only criterion for when a separating set is a generating set. The second is the introduction of a class of generalized polarizations which, in the case of representations of finite groups, always generates the algebra of invariants of their diagonal representations.
\end{abstract}

\maketitle \todo{reverted title back to title in v1} \todo{changed abstract a bit, closer to v1}

%\tableofcontents

\section{Introduction}
%\todo{I slightly modified your introduction, with Harm Derksen as my ideal audience}
\subsection*{Invariant Theory without groups}

Given a representation $V$ of a group $G$, Invariant Theory studies the algebra of polynomials $f$ on $V$ which are invariant under the $G$-action $(g.f)(x)=f(g^{-1}x)$. 
%\todo{I tried to get to the point as fast as possible}

We will restrict ourselves to the case where $V$ is a real, finite-dimensional inner product space (a \emph{Euclidean} vector space, for short), and $G$ is a compact group acting by orthogonal transformations. In this case, the algebra $\R[V]^G$ of invariant polynomials can be thought of as the algebra of polynomial functions $V\to \R$ which are constant along the fibers of the quotient map $\pi:V\to V/G$. Moreover, $\R[V]^G$ separates orbits, in particular the orbit space $V/G$ is Hausdorff.

% there is a geometric quotient $V/G$, the quotient map $\pi:V\to V/G$ separates orbits, and the algebra $\R[V]^G$ of invariant polynomials can be thought as the algebra of polynomials which are constant along the fibers of $\pi$.

This picture can be generalized to a setup not involving any group at all: notice, in fact, that the fibers of $\pi:V\to V/G$ (i.e. the $G$-orbits) are smooth, embedded and pair-wise equidistant submanifolds of $V$. Using these properties as a definition, one arrives at the notion of \emph{infinitesimal manifold submetry} $\sigma:V\to X$, from $V$ onto a metric space $X$, see \cite{MR19} for a precise definition (which we also recall in Section \ref{S:submetries} for the reader's convenience).
In addition to being a generalization of orbit decompositions, the fibers of manifold submetries also generalize classical objects from Foliation Theory, such as: \emph{transnormal systems}; \emph{singular Riemannian foliations}; and \emph{isoparametric foliations}, that is, the foliation given by the parallel and focal submanifolds of isoparametric submanifolds. In particular, there exist many important examples of infinitesimal manifold submetries which are \emph{inhomogeneous}, that is, not given by the orbits of some orthogonal action by a compact group, see Subsection \ref{SS:homogeneity}.

The ``Invariant Theory'' of infinitesimal manifold submetries has its roots in the study, by many authors,  of the algebraic aspects of singular Riemannian foliations and isoparametric foliations of Euclidean space. It starts with the definition of the algebra $\mathcal{B}(\sigma)$ of polynomials constant along the fibers of an infinitesimal manifold submetry $\sigma$, called \emph{basic polynomials}, generalizing the algebra $\R[V]^G$ of invariant polynomials. Crucially, $\B(\sigma)$ separates fibers, and thus determines $\sigma$, like in the homogeneous case. An early version of this was proved for isoparametric foliations in \cite{Muenzner80, Muenzner81}, then for singular Riemannian foliations in \cite{LR18}, and later for general manifold submetries in \cite{MR19}.

%The interaction between infinitesimal manifold submetries $\sigma:V\to X$ and their algebra $\mathcal{B}(\sigma)$ was studied in \cite{MR19}. 
One of the goals of the present paper is to show how this more general theory can give back to Classical Invariant Theory. 
Some of the applications we present (Corollaries \ref{MC:separating},  \ref{MC:polarization}, and  Theorem \ref{MT:maximality}) have purely algebraic proofs, and thus could be proved without leaving the realm of Classical Invariant Theory. Others (Corollary \ref{MC:IIT} and Theorem \ref{MT:polarization2}) rely heavily on the geometric machinery of \cite{MR19} and previous works. %, and the authors cannot envision proofs that avoid manifold submetries.

\subsection*{Maximality of Laplacian algebras} \todo{moved this subsection here}
Define \emph{Laplacian algebras}, as the sub-algebras $A\subset \R[V]$ that contain the  ``distance-squared'' polynomial $r^2=\sum_i x_i^2$ and are preserved by the differential operator ``dual'' to  $r^2$, namely the Laplacian $\Delta = \sum_i \partial^2/\partial x_i^2$. 
To define maximality, note that any $A\subset\R[V]$ defines  an equivalence relation $\sim_A$ on $V$ where $x\sim_A y$ if and only if $f(x)=f(y)$ for all $f\in A$. Then we say that a sub-algebra $A\subset\R[V]$  is \emph{maximal} if any strictly larger algebra would define an equivalence relation strictly finer than $\sim_A$.

It was conjectured in \cite{MR19} that Laplacian implies maximal, and the solution of this conjecture is the technical heart of the present paper:
\begin{maintheorem}
\label{MT:maximality}
Let $A\In\R[V]$ be a Laplacian algebra. Then $A$ is maximal.
\end{maintheorem}

\subsection*{Inverse Invariant Theory} One indication that infinitesimal manifold submetries are a natural setting for Invariant Theory is that it allows for the solution of problems that seem to have been  out of reach of classical Invariant Theory. Take, for example, the \emph{Inverse Invariant Theory Problem}, about characterizing the sub-algebras of $\R[V]$ which are algebras of invariants of some representation.  This seems to only have been solved over finite fields \cite[Section 8.4]{NeuselSmith}, but in our more general context there is a satisfying solution, namely the Laplacian algebras defined above:
\begin{maincor}
\label{MC:correspondence} \todo{a corollary instead of theorem}
Let $V$ be a Euclidean vector space. Then, taking algebras of basic polynomials gives a one-to-one correspondence between infinitesimal manifold submetries from $V$, and Laplacian sub-algebras of $\R[V]$. 
\end{maincor}
Corollary \ref{MC:correspondence} is an immediate consequence of Theorem \ref{MT:maximality} and \cite[Theorem A]{MR19}. Indeed, the only difference between Corollary \ref{MC:correspondence} and \cite[Theorem A]{MR19}, is that, in the latter, the algebras are required to be \emph{both} Laplacian and maximal.

Analogously, the maximality condition can be removed from \cite[Theorems B and C]{MR19}, which provides the following solution to the Inverse Invariant Theory Problem for finite groups:
\begin{maincor}\label{MC:IIT} \todo{a corollary instead of theorem, and added part (b)}
%\todo{I think this is better written as a theorem, since it is also a result we obtain using the invariant theory for submetries}
Let $A\In\R[V]$ be a sub-algebra. Then
\begin{enumerate}[(a)]
\item $A$ is the algebra of invariants of a finite subgroup $G\In\OO(V)$ if and only if $A$ is Laplacian and the field of fractions of $A$ has transcendence degree (over $\R$) equal to $\dim(V)$.
\item $A$ is the algebra of basic polynomials of a transnormal system if and only $A$ is Laplacian and integrally closed in $\R[V]$.
\end{enumerate}
\end{maincor}

%Analogously, the word ``maximal'' can also be removed from the statements of \cite[Theorems B and C]{MR19}:
%\begin{corollary} \label{C:correspondence}
%Let $A\In\R[V]$ be a sub-algebra. Then
%\begin{enumerate}[(a)]
%\item $A$ is the algebra of invariants of a finite subgroup $G\In\OO(V)$ if and only if $A$ is Laplacian and the field of fractions of $A$ has transcendence degree (over $\R$) equal to $\dim(V)$.
%\item $A$ is the algebra of basic polynomials of a transnormal system if and only $A$ is Laplacian and integrally closed in $\R[V]$.
%\end{enumerate}
%\end{corollary}

\subsection*{Application to separating sets}

Another  application to Classical Invariant Theory concerns \emph{separating sets}, which is a topic of much recent research --- see \cite[Section 2.4]{DerksenKemper} and references therein. In our context, a set $\mathcal{S}$ of $G$-invariant polynomials is called separating if it separates $G$-orbits.
%In general, given a $G$-representation $V$, a set $\mathcal{S}$ of polynomials is called separating if it separates any two points which can be separated by invariant polynomials. In our situation $G$ is compact, so that just means two points in different orbits. 
The following characterizes the separating sets which generate the whole algebra of invariants:
\begin{maincor}
\label{MC:separating} \todo{a corollary instead of theorem}
Let $V$ be a real orthogonal representation of the compact group $G$, and let $\R[V]^G$ denote its algebra of invariants. Let $\mathcal{S}\In \R[V]^G$ be a separating set containing $r^2$, and $B$ the subalgebra generated by $\mathcal{S}$. Then $B=\R[V]^G$ if and only $\Delta f$ and $\scal{\nabla f, \nabla g}$ belong to $B$, for every $f,g\in \mathcal{S}$.
\end{maincor}
We note that the condition on $\mathcal{S}$ given in Corollary \ref{MC:separating} is equivalent to $B$ being Laplacian (see \cite[Proposition 37]{MR19} or Lemma \ref{L:Lapiff} below), so that Corollary \ref{MC:separating} is a direct consequence of Theorem \ref{MT:maximality}. With a bit more work, we in fact prove a stronger version, which holds for infinitesimal manifold submetries, and only requires $\mathcal{S}$  to be ``locally'' separating (see Theorem \ref{T:separating}). For a different criterion for a separating set to be generating, in the context of rational representations of reductive groups over algebraically closed fields, see \cite[Theorem 2.4.6]{DerksenKemper}.

%\begin{lemma}
%\label{L:Lapiff}
%Let $\mathcal{S}\In\R[V]$ be a subset containing $r^2$. Then the sub-algebra $B\In\R[V]$ generated by $\mathcal{S}$ is Laplacian if and only if $\Delta f$ and $\scal{\nabla f, \nabla g}$ belong to $B$, for every $f,g\in \mathcal{S}$.
%\end{lemma}
%\begin{proof}
%See \cite[Proposition 37]{MR19}.
%\end{proof}

Corollary \ref{MC:separating} can be a useful tool to prove First Fundamental Theorems, that is, to show that certain sets of polynomials generate the algebra of basic polynomials of a given manifold submetry, see Subsection \ref{SS:FFT} for a few illustrative examples.

\subsection*{Application to polarizations}

Taking the sum of $k$ copies of a $G$-representation $V$ produces a $G$-representation $V^k$. Recall that, given a homogeneous $G$-invariant polynomial $f$ of degree $d$, its \emph{polarizations} (which we will sometimes call \emph{classical polarizations} to distinguish them from a generalization described below)  are the multi-variable invariants $f_\alpha\in\R[V^k]^G$, where $\alpha=(\alpha_1, \ldots , \alpha_k)$ runs through the multi-indices with $|\alpha| =\sum_i \alpha_i=d$, defined by
\begin{equation}\label{E:polarization}
f\left(\sum_i s_i v_i\right)=\sum_{|\alpha|=d} s_1^{\alpha_1}\cdots s_k^{\alpha_k} f_\alpha(v_1, \ldots, v_k)
\end{equation}
where $s_i$ are formal variables.

Alternatively, the algebra generated by all polarizations can also be seen as the smallest sub-algebra of $\R[V^k]$ containing $\R[V]^G$ (seen as polynomials depending only on the first variable $v_1$) and closed under a certain family of differential operators called \emph{polarization operators}, see Subsection \ref{SS:poldefs}.

Inspired by the latter form of the definition, we introduce a sub-algebra $A^{(k)}\subset \R[V^k]$, which we call the algebra of \emph{generalized polarizations}, associated to any Laplacian algebra $A\subset\R[V]$. It is defined as the smallest Laplacian algebra containing $A$ (seen as polynomials depending only on the first variable $v_1$) and the inner products $(v_1, \ldots, v_k)\mapsto \langle v_i, v_j \rangle$ for all $i,j$.  When $A=\R[V]^G$ for a compact $G$, the algebra $A^{(k)}$ contains all classical polarizations, and is contained in $\R[V^k]^G$. 

For most $G$-representations $V$, classical polarizations \emph{do not} generate  $\R[V^k]^G$. It is conjectured, for example, that for $G$ finite, $\R[V^2]^G$ can only be generated by classical polarizations if $G$ is generated by reflections, see  \cite{Schwarz07} and Remark \ref{R:polproperty} below for a more complete discussion. In contrast, one has:
\begin{maincor}
\label{MC:polarization} \todo{a corollary instead of theorem}
Let $V$ be a real orthogonal representation of the finite group $G$, and $k\geq 2$. Then the algebra of invariants of the $G$-representation $V^k$ is generated by generalized polarizations. 
\end{maincor}

Corollary \ref{MC:polarization} is an immediate consequence of Corollary \ref{MC:separating} and the fact that classical polarizations separate $G$-orbits in $V^k$ for finite $G$, see \cite[Theorem 3.4]{DKW08}. 

The idea of using the Laplacian was also employed in \cite{Iltyakov98} to find generators for the algebra $\mathbb{C}[V^k]^G$ in the special case where $V$ is a complex finite-dimensional simple algebra with a non-degenerate inner product and a real form, and $G$ is the automorphism group of $V$.

The authors suspect that taking the diagonal representation does not have an analogue for inhomogeneous infinitesimal manifold submetries, and, relatedly, that the algebra of generalized polarizations $A^{(k)}$ is homogeneous for every $k\geq 2$ and every Laplacian algebra $A$. The second main result in the present paper is a partial result in this direction, under the additional assumption (called ``$k$-NS'', see Definition \ref{D:k-NS} and Lemma \ref{L:k-NS}) that, for generic $x_1,\ldots , x_k\in V$, the normal spaces to the $\sigma$-fibers at these points span $V$:
\begin{maintheorem}[Homogeneity of generalized polarizations]
\label{MT:hompol} \todo{added as main theorem}
Let $\sigma:V\to X$ be an infinitesimal manifold submetry satisfying $k$-NS, with associated algebra of basic polynomials $A=\B(\F)$. Let $\OO(\sigma)$ be the closed subgroup of $\OO(V)$ consisting of all $g\in\OO(V)$ that map each $\sigma$-fiber to itself, and consider the diagonal action of $\OO(\sigma)$ on $V^k$. Then $A^{(k)}=\R[V^k]^{\OO(\sigma)}$.
\end{maintheorem}
Because of Corollary \ref{MC:separating} (more precisely, its local version Theorem \ref{T:separating}), the proof of Theorem \ref{MT:hompol} is reduced to showing that $A^{(k)}$ locally separates $\OO(\F)$-orbits. Thus Theorem \ref{MT:hompol} is analogous to \cite[Theorem 3.4]{DKW08}, and can replace its use in the proof of Corollary \ref{MC:polarization}, see Remark \ref{R:altproof}.

To discuss how Theorem \ref{MT:hompol} leads to a generalization of  Corollary \ref{MC:polarization} for certain $G$-representations $V$, where $G$ is an infinite compact group, we assume for simplicity that the representation is faithful, so that we may treat $G$ as a subgroup of the orthogonal group $\OO(V)$.

Since $G$ is infinite, one faces a new difficulty in that a different subgroup $G'$ may have the same orbits as $G$, so that in particular $\R[V]^G=\R[V]^{G'}$. Such a pair of groups are called \emph{orbit-equivalent}, the simplest example being $\SO(n)$ and $\OO(n)$ acting on $\R^n$ for $n>1$. In such a case, $\R[V^k]^G$ and $\R[V^k]^{G'}$ will necessarily be different for $k$ large enough. Therefore, no procedure that produces elements of $\R[V^k]^G$ out of $\R[V]^G$ can possibly be enough to generate all of $\R[V^k]^G$ for a general subgroup $G$. To fix this, we impose the condition that $G$ be maximal (with respect to inclusion) in its orbit-equivalence class. In the notation of Theorem \ref{MT:hompol}, this is equivalent to $G=\OO(\sigma)$, where $\sigma$ is the natural map $V\to V/G$. Thus we immediately obtain the first part of the Theorem below, where we also list two natural conditions that imply $k$-NS:

\begin{maintheorem}
\label{MT:polarization2} \todo{a corollary instead of theorem, and added general statement involving $k$-NS}
Let $G\In \OO(V)$ be a closed subgroup, with algebra of invariants $A=\R[V]^G$. Assume that $G$ is maximal in its orbit-equivalence class, and that the decomposition of $V$ into $G$-orbits satisfies $k$-NS, for $k\geq 2$. Then the algebra of invariants of the $G$-representation $V^k$ coincides with the algebra of generalized polarizations $A^{(k)}$.

Moreover, either condition below implies $k$-NS:
\begin{enumerate}[(a)]
\item the connected component of $G$ is a torus;
\item $k\geq \dim(V)$;
\end{enumerate}
\end{maintheorem}

The authors do not know if the $k$-NS condition is necessary:
\begin{question}
Let $G$ be a compact subgroup of $\OO(V)$, maximal in its orbit-equivalence class,  and $k\geq 2$. Is the algebra of invariants of the $G$-representation $V^k$ is generated by generalized polarizations?
\end{question}

Finally, we mention that generalized polarizations can be used to give a sufficient criterion for a manifold submetry to be homogeneous, see Subsection \ref{SS:homogeneity}.

\subsection*{Sketch of the proofs}

To prove  that a Laplacian algebra $A$ is maximal (Theorem \ref{MT:maximality}), we consider the unit sphere $SV$ in $V$, and its quotient $X=SV/\sim_A$ by the equivalence relation given by $A$ as above. Then $A$ is an algebra of continuous functions on the compact Hausdorff topological space $X$, which separates points of $X$ by definition, so the Stone--Weierstrass Theorem implies that $A$ is dense in $C^0(X)$.

Assuming $A$ is not maximal, there is a homogeneous polynomial $f\notin A$ that is constant on the equivalence classes of $\sim_A$, and hence descends to a continuous function on $X$. The fact that $A$ is Laplacian means that $A$ is compatible with the Theory of Spherical Harmonics. This implies that $f$ may be taken orthogonal to $A$ in the appropriate sense, and that the induced element of $C^0(X)$ is orthogonal to $A$ in the $L^2$ inner product, contradicting density of $A$.

%Corollaries \ref{MC:correspondence}, \ref{MC:IIT}, and \ref{MC:separating} follow  from  Theorem \ref{MT:maximality} and  \cite{MR19}, and Corollary \ref{MC:polarization} follows from Corollary \ref{MC:separating} and \cite[Theorem 3.4]{DKW08}.

We turn to Theorem  \ref{MT:hompol}. By a local version of Corollary \ref{MC:separating} (Theorem \ref{T:separating} below), it is enough to show that $A^{(k)}$ generically separates $\sigma$-fibers, because $A^{(k)}$ is Laplacian by definition.

For simplicity, assume $k=2$. Given $(x,y), (z,w)\in V^2$ such that $F(x,y)=F(z,w)$ for all $F\in A^{(2)}$, we need to find $g\in \OO(\sigma)$ such that $g(x,y)=(z,w)$. Since $A^{(2)}$ contains all inner products by definition, we have $\|x\|=\|z\|$, $\|y\|=\|w\|$, and $\langle x, y \rangle= \langle z,w\rangle$, and so there exists $g\in \OO(V)$ such that $g(x,y)=(z,w)$. Making $F$ run through all classical polarizations, the original definition given by \eqref{E:polarization} shows that $f(sx+ty)=f(sz+tw)$ for all $f\in A$, so that $sx+ty$ and $g(sx+ty)$ belong to the same $\sigma$-fiber, for all $s,t\in \R$.

A similar argument using generalized polarizations instead of classical polarizations yields $g\in \OO(V)$ such that $g(x,y)=(z,w)$, and such that $\nabla f (x)+ \nabla h(y)$ and $g(\nabla f (x)+ \nabla h(y))$ belong to the same $\sigma$-fiber for all $f,h\in A$. The $2$-NS condition implies that, for $(x,y)$  generic,  every $v\in V$ can be written in the form $\nabla f (x)+ \nabla h(y)$ for appropriate choices of $f,h$. Thus $g\in\OO(V)$ takes every $\sigma$-fiber to itself, that is,  $g\in \OO(\sigma)$.

\subsection*{Organization of the paper}
In Section \ref{S:maximality} we recall some basic definitions, and the theory of Spherical Harmonics, and give a proof of Theorem \ref{MT:maximality}. Section \ref{S:submetries} contains definitions and basic facts involving submetries. In Section \ref{S:separating} we introduce the notion of local separating set, prove Theorem \ref{T:separating} (a strengthening of Corollary \ref{MC:separating}), and illustrate how it can be used to prove First Fundamental Theorems. Section \ref{S:pols} is devoted to the study of generalized polarizations. It includes the proofs of Theorem \ref{MT:hompol}, Theorem \ref{MT:polarization2}, and of a sufficient criterion for a manifold submetry to be homogenous.

\subsection*{Acknowledgements} 
It is a pleasure to thank Harm Derksen for pointing out \cite{Schrijver08}, which was the main inspiration for the proof of Theorem \ref{MT:maximality}, and Matyas Domokos for a simplification of the proof of Corollary \ref{MC:polarization} using \cite{DKW08}. We would also like to thank Alexander Lytchak and Matyas Domokos for suggestions that improved the exposition. %\todo{ask Harm if it's ok to mention him}

%%%%%%%%%%%%%%%%%%%%%%%%%%%%%%%%%%%%
\section{Laplacian algebras are maximal}
\label{S:maximality}
This  section is devoted to the proof of Theorem \ref{MT:maximality}, given in Subsection \ref{SS:maximality}. To fix notations and for the sake of completeness, the definitions and basic facts needed in the proof are laid out in the first four subsections. The material in Subsections \ref{SS:polynomials} and \ref{SS:harmonic} is well-known, see for example \cite[Exercise 12(e), page 118]{HoweTan} or \cite[Introduction, Section 3]{HelgasonGGA}. The material in Subsections \ref{SS:laplacian} and \ref{SS:maximal} is either contained in, or follows easily from, \cite{MR19}.

\subsection{Polynomials}\label{SS:polynomials}
Let $V$ be a Euclidean vector space, that is, a real finite-dimensional vector space with inner product $\scal{\,,}$ (which we will occasionally write $\scal{\,,}_V$), and denote by $\OO(V)$ the corresponding orthogonal group.

Let $\R[V]$ be the algebra of polynomial functions $V\to\R$, which is graded in the sense that 
\[ \R[V]=\bigoplus_{d=0}^\infty \R[V]_d \]
where $\R[V]_d$ denotes the space of homogeneous polynomials of degree $d$. The group $\OO(V)$ acts on $\R[V]$ by $(Uf)(v)= f(U^{-1}v)$ for $f\in \R[V]$ and $U\in\OO(n)$, and this action preserves the grading.

Choose an orthonormal basis $e_1, \ldots e_n$ of $V$, and dual basis $x_1, \ldots x_n$ of $V^*$. In particular we have an identification of $\R[V]$ with $\R[x_1, \ldots , x_n]$, and of $\OO(V)$ with the group of orthogonal $n\times n$ matrices $\OO(n)$.

Given $f\in\R[V]_d$, its \emph{dual} is the differential operator $\hat{f}$, of order $d$, obtained from $f$ by replacing each variable $x_i$ with the partial derivative $\frac{\partial}{\partial x_i}$, and products with composition. Define a bilinear form $\scal{\ ,}_d$ on each $\R[V]_d$ by
\begin{equation}
\label{E:innerproduct}
\scal{f,g}_d= \hat{f}(g). 
\end{equation} 
It is not hard to see that this is an inner product,  that the monomials $x_1^{\alpha_1} \ldots x_n^{\alpha_n}$ (where $\alpha_i\geq 0$ and $\alpha_1+\cdots \alpha_n=d$) form an orthogonal basis, and that $\|x_1^{\alpha_1} \cdots x_n^{\alpha_n}\|_d^2=\alpha_1!\cdots \alpha_n!$. Note also that the dual operation satisfies $\widehat{fg}=\hat{f}\hat{g}$, so in particular multiplication with a polynomial $h$ is adjoint to $\hat{h}$. More precisely, for every triple of homogeneous polynomials with $\deg(fh)=\deg(g)$, one has $\scal{hf,g}_{\deg(g)}=\scal{f,\hat{h}(g)}_{\deg(f)}$.

\begin{lemma}
The inner product $\scal{\ ,}_d$  on $\R[V]_d$ defined in \eqref{E:innerproduct} above is $\OO(V)$-invariant. 
\end{lemma}
\begin{proof}
$\R[V]_d$ is isomorphic, as an $\OO(V)$-representation, to the space of symmetric tensors $\Sym^d(V^*)$, that is, to the subspace of all elements of $(V^*)^{\otimes d}$ that are fixed by the natural action of the permutation group $S_d$. Namely, the multilinear map  $\alpha\in \Sym^d(V^*)$ corresponds to the polynomial function $f=\phi(\alpha)$ given by $f(v)=\alpha(v, \ldots, v)$.

The inner product on $V$ induces a natural inner product on $V^*$, which we also denote by $\scal{\,,}$, which then induces the following inner product on $(V^*)^{\otimes d}$, which is clearly $\OO(V)$-invariant:
\[ \scal{ \lambda_1\otimes \cdots \otimes \lambda_d , \mu_1\otimes \cdots\otimes \mu_d}= \scal{\lambda_1, \mu_1} \cdots \scal{\lambda_d, \mu_d} .\] 

Thus it suffices to show that $\scal{\phi(\alpha), \phi(\beta)}_d=d!\scal{\alpha,\beta}$ for all $\alpha, \beta\in\Sym^d(V^*)$. To this end, note that $x_1^{\alpha_1} \ldots x_n^{\alpha_n}=\phi(\alpha)$ with 
\[\alpha= \frac{1}{d!}\sum_{\sigma\in S_d}\sigma\big((x_1)^{\otimes\alpha_1}\otimes\cdots \otimes (x_n)^{\otimes\alpha_n}\big)\]
and that $\scal{\alpha,\alpha}=$
\[\frac{1}{(d!)^2}\sum_{\sigma,\tau\in S_d}\scal{\sigma\big((x_1)^{\otimes\alpha_1}\otimes\cdots \otimes (x_n)^{\otimes\alpha_n}\big), \tau \big((x_1)^{\otimes\alpha_1}\otimes\cdots \otimes (x_n)^{\otimes\alpha_n}\big)}\]
which equals $\frac{\alpha_1!\cdots\alpha_n!}{d!}$.
\end{proof}

\subsection{Harmonic polynomials}\label{SS:harmonic}
We keep the same notations as in the previous subsection. Let $r^2$ denote the quadratic polynomial
\[r^2= x_1^2+\cdots +x_n^2\in\R[V],\]
and let $\Delta$ denote the Laplace operator
\[\Delta=\widehat{r^2}=\frac{\partial^2}{\partial x_1^2} +\cdots+\frac{\partial^2}{\partial x_n^2}.\]
For each $d\geq 0$, let $\Har_d$ denote the subspace of $\R[V]_d$ consisting of the \emph{harmonic polynomials}, that is, polynomials $f$ satisfying $\Delta(f)=0$. Then one has the following $\scal{\ ,}_d$-orthogonal  direct sum decomposition of $\R[V]_d$. 
\begin{equation}
\label{E:harmonic}
\R[V]_d=\Har_d\oplus r^2\Har_{d-2}\oplus\cdots\oplus r^{2\lfloor d/2 \rfloor}\Har_{d-2\lfloor d/2\rfloor}
\end{equation}
The summands above are also the irreducible components of $\R[V]_d$ as an $\OO(V)$-representation, and they are pairwise inequivalent. These facts are usually called the theory of spherical harmonics, see for example \cite[Exercise 12(e), page 118]{HoweTan}.

%see \ref{FultonHarris}[Exercise 19.21].

Another, more geometric, inner product one can define on $\R[V]_d$ is the $L^2$-product given by
\[\scal{f,g}_{L^2} = \int_{SV}fg\, d\mathrm{vol}\]
where $SV$ denotes the unit sphere in $V$, and $d\mathrm{vol}$ its natural Riemannian volume form. Since this inner product is also $\OO(V)$-invariant, we can apply Schur's Lemma to conclude that there exist positive constants $C_i^d$ such that
\begin{equation}
\label{E:L2}
\scal{f,g}_{L^2}=\sum_{i=0}^{\lfloor d/2\rfloor}C_i^d \scal{f_i, g_i}_d
\end{equation} 
according to the decomposition of $f,g\in\R[V]_d$ in \eqref{E:harmonic}. That is, $f=\sum_i f_i$, and $g=\sum_i g_i$, with $f_i, g_i\in r^{2i}\Har_{d-2i}$.

\subsection{Laplacian algebras}\label{SS:laplacian}
Laplacian algebras were introduced in \cite{MR19}:
\begin{definition}
Let $A\In\R[V]$ be a sub-algebra. It is called \emph{Laplacian} if it contains $r^2$ and is preserved by the Laplace operator $\Delta$.
\end{definition}

\begin{lemma}
\label{L:Agraded}
Let $A\in\R[V]$ be a Laplacian algebra. Then 
\begin{enumerate}[(a)]
\item $A$ is graded, that is,
\[ A=\bigoplus_{d=0}^\infty A_d ,\]
where $A_d=A\cap\R[V]_d$.

\item Each $A_d$ is graded with respect to the decomposition in \eqref{E:harmonic}:
\[ A_d=(A_d\cap\Har_d)\oplus (A_d\cap r^2\Har_{d-2})\oplus\cdots\oplus (A_d\cap r^{2\lfloor d/2 \rfloor}\Har_{d-2\lfloor d/2\rfloor}).\]

%\item $A$ is preserved by $\hat{f}$, for every $f\in A$.
\end{enumerate}
\end{lemma}
\begin{proof}
\begin{enumerate}[(a)]
\item
Since $A$ is preserved by $\Delta$ and by multiplication with $r^2$, it is preserved by the Lie bracket $[\Delta, r^2]$ of these two linear maps. The latter is a linear endomorphism of $\R[V]$ with eigenspaces $\R[V]_d$, so  $A$ is graded.

\item Note that the $\scal{\ ,}_d$-orthogonal  complement of $r^2A_{d-2}$ in $A_d$ is exactly $A_d\cap\Har_d$. Indeed, $f\in A_d$ is orthogonal to  $r^2A_{d-2}$ if and only if $\Delta(f)=0$, because $\Delta(f)$ is an element of $A_{d-2}$ that satisfies 
$ \scal{\Delta(f),g}_{d-2}=\scal{f, r^2 g}_d$ for all $g\in A_{d-2}$. The result now follows by induction. 

%\item See \cite{MR19}[Corollary 22.(3)]
\end{enumerate}
\end{proof}

\subsection{Partitions of vector spaces and maximal algebras}\label{SS:maximal}

Given a sub-algebra $A\In\R[V]$ (or, in fact, any subset), define the equivalence relation $\sim_A$ on $V$ by
\[ v\sim_A w \iff f(v)=f(w) \quad \forall f\in A\]
In words, $v,w$ are equivalent if and only if they cannot be separated by any element of $A$.

Denote by $\Le(A)$ the partition of $V$ into the equivalence classes (also called \emph{leaves}) of $\sim_A$. The symbol $\Le$ stands for {\it L}evel sets, because the elements of $\Le(A)$ are the common level sets of polynomials in $A$.

In the opposite direction, given a partition $\F$ of $V$, we define the sub-algebra $\B(\F)\In\R[V]$ as the algebra  of all ($\F$-)$\emph{basic}$ polynomials, that is, polynomials that are constant on the leaves of $\F$.

Given a partition $\F$ of $V$, a subset of $V$ is called \emph{$\F$-saturated} if it is a union of $\F$-leaves. 
Partially order partitions by $\F < \F'$ when $\F$ is \emph{coarser} than $\F'$, that is, when every $\F$-leaf is $\F'$-saturated. On the other hand, we partially order sub-algebras of $\R[V]$ by inclusion. Then both $\Le$ and $\B$ preserve these partial orders. Moreover, one has the tautologies $A\In\B(\Le(A))$ and $\Le(\B(\F))<\F$. In particular, $\B\circ\Le\circ\B=\B$ and $\Le\circ\B\circ\Le=\Le$.

\begin{definition}[\cite{MR19}]
A sub-algebra $A\In\R[V]$ is called \emph{maximal} if $A=\B(\Le(A))$, or, equivalently, if it is the algebra of basic polynomials of some partition of $V$.
\end{definition}

\begin{lemma}
\label{L:Bgraded}
Let $A\In\R[V]$ be a graded sub-algebra. Then $B=\B(\Le(A))$ is also graded.
\end{lemma}
\begin{proof}
Since $A$ is graded, the homothetic transformations $h_\lambda:V\to V$ (defined by $h_\lambda(v)=\lambda v$) send $\Le(A)$-leaves onto $\Le(A)$-leaves, for all $\lambda\in\R\setminus\{0\}$.

Let $f\in B$ be a polynomial of degree $d$, and, for each $i=0,\ldots, d$, let $f_i$ be its homogeneous component of degree $i$. Thus $f=\sum_{i=0}^d f_i$ and so 
\[ \lambda^{-d} f\circ h_\lambda= \lambda^{-d}f_0+\cdots+\lambda^{-1} f_{d-1}+f_d \in B \] 
for all $\lambda\neq 0$. Taking $\lambda\to\infty$ shows that  $f_d\in B$, so that $f-f_d\in B$, and one can proceed by induction on the degree of $f$.
\end{proof}

\subsection{Laplacian algebras are maximal}
\label{SS:maximality}

\begin{proof}[Proof of Theorem \ref{MT:maximality}]
Let $B=\B(\Le(A))$. Assume for a contradiction that the inclusion $A\In B$ is strict. By Lemmas \ref{L:Agraded} and \ref{L:Bgraded}, the algebras $A$ and $B$ are graded, so there exists $d$ such that the inclusion $A_d\In B_d$ is strict. Let $f\in B_d\setminus\{0\}$ orthogonal to $A_d$ with respect to the inner product $\scal{\ ,}_d$ defined in \eqref{E:innerproduct}. 

Let $X=SV/\!\sim_A$ be the quotient topological space of the unit sphere $SV\In V$ by the equivalence relation $\sim_A$, which restricts to $SV$ because $r^2\in A$. Since $SV$ is compact, so is $X$. Let $C^0(X)$ denote the algebra of continuous real-valued functions on $X$, and consider the algebra homomorphism 
\[\varphi: B\to C^0(X)\]
where, given $h\in B$,  $\varphi(h)$ is  the function on $X$ induced by the restriction $h|_{SV}:SV\to \R$.

Then $\varphi(A)$ is a sub-algebra of $C^0(X)$ which separates points, and contains the constant functions. By the Stone--Weierstrass Theorem (see \cite[Theorem 7.32 on page 162]{Rudin}) $\varphi(A)$ is dense in $C^0(X)$ with respect to the supremum norm.

Since $f$ is homogeneous  and non-zero, we have $\int_{SV} f^2\, d\mathrm{vol}>0$. Choose $g\in A$ such that the $C^0$-distance
\[ \sup_{x\in X}| \varphi(g)(x) -\varphi(f) (x)| = \sup_{v\in SV} |g(v)-f(v)|\]
is small enough so that 
\begin{equation}
\label{E:L2>0}
\int_{SV} fg\, d\mathrm{vol}>0.
\end{equation}

Assume $d=\deg{f}$ even. (The case $d$ odd is analogous and is left to the reader.)
Since $f|_{SV}$ is an even function, we may assume $g$ is also even, that is, has only even degree homogeneous components. Indeed, $A$ is graded so that the even part of $g$ belongs to $A$, and the integral in \eqref{E:L2>0} does not change after replacing $g$ with its even part.
%because the $C^0$-distance between $f|_{SV}$ and the even part of $g|_{SV}$ is at most the $C^0$-distance between $f|_{SV}$ and $g|_{SV}$.
Moreover, by multiplying the homogeneous components of $g$ with appropriate powers of $r^2$, we may assume $g$ is a homogeneous element of $A$, $\deg(g)$ is an even number $\geq d$, and \eqref{E:L2>0} still holds.

On the other hand, for any $k$, the polynomial $r^{2k} f$ is a non-zero element of $B_{d+2k}$ which is $\scal{\ ,}_{d+2k}$-orthogonal  to $A_{d+2k}$, because $A$ is Laplacian. Thus, after replacing $f$ with $r^{2k} f$ for an appropriate choice of $k$ (so that \eqref{E:L2>0} still holds), we may assume $f,g$ are homogeneous of the same (even) degree $d$.

Decompose $f, g$ with respect to \eqref{E:harmonic}. That is, $f=\sum_i f_i$, and $g=\sum_i g_i$, with $f_i, g_i\in r^{2i}\Har_{d-2i}$. By Lemma \ref{L:Agraded}, $g_i\in A_d$ for every $i$. Since $f$ is orthogonal to $A_d$, Lemma \ref{L:Agraded} implies that each $f_i$ is orthogonal to $A_d$. Thus, by \eqref{E:L2}, we obtain $\int_{SV}fg\, d\mathrm{vol}=0$, contradicting \eqref{E:L2>0}.
\end{proof}

\section{Manifold submetries}\label{S:submetries}
For the reader's convenience, and to fix notations, we recall definitions and basic facts regarding submetries. For more information, see \cite{KL20, MR19} and references therein.

\begin{definition}\ 
\begin{itemize}
\item
A \emph{submetry} is a map between metric spaces which maps closed metric balls to closed metric balls of the same radius.
\item A \emph{manifold submetry} is a submetry from a Riemannian manifold to a metric space, such that each fiber is a possibly disconnected embedded smooth submanifold.
\item A \emph{spherical manifold submetry} is a manifold submetry from the unit sphere $SV$ in a Euclidean vector space $V$.
\item An \emph{infinitesimal manifold submetry} is a manifold submetry from a Euclidean vector space $V$, such that the origin is a fiber.
\end{itemize}
\end{definition}
We note the fibers of a submetry form a partition of the domain into closed equidistant subsets, and that, conversely, any such partition comes from a submetry. We also note that, given an infinitesimal manifold submetry $\sigma:V\to X$, the unit sphere $SV$ is a union of fibers, and the restriction $\sigma|_{SV}:SV\to \sigma(SV)$ is a spherical manifold submetry. Moreover, this procedure establishes a one-to-one correspondence between these two types of manifold submetries, see \cite[Appendix B1]{MR19}.

\begin{example}
Let $G$ be a compact group, and $V$ be an orthogonal $G$-representation. Then the quotient map $V\to V/G$ is an infinitesimal manifold submetry, and $SV\to SV/G$ is a spherical manifold submetry.
\end{example}

\begin{definition}
Given an infinitesimal manifold submetry $\sigma:V\to X$, denote by $\F_\sigma$ the partition of $V$ into the fibers of $\sigma$, and define its \emph{algebra of basic polynomials} by $\B(\sigma)=\B(\F_\sigma)$.
\end{definition}

We will frequently abuse notation and write ``let $(V,\F)$ be a manifold submetry'' when we mean ``let $\sigma:V\to X$ be a manifold submetry, and $\F=\F_\sigma$''.
\todo{removed proof of Corollary \ref{MC:correspondence}, and moved old Corollary 10 (now Corollary \ref{MC:IIT}) to the introduction}
%\begin{proof}[Proof of Corollary \ref{MC:correspondence}]
%By \cite[Theorem A]{MR19}, $\sigma\mapsto\B(\sigma)$ gives a one-to-one correspondence between infinitesimal manifold submetries (modulo the equivalence relation of having the same fibers), and maximal Laplacian sub-algebras of $\R[V]$. By Theorem \ref{MT:maximality}, the maximality condition is superfluous, because it follows from Laplacian.
%\end{proof}

\section{Separating versus generating invariants}
\label{S:separating}
In this section we prove a result (Theorem \ref{T:separating} below) that generalizes Corollary \ref{MC:separating} in two directions: the orthogonal representation is replaced with a manifold submetry, and the set $\mathcal{S}$ is only assumed to separate leaves on a certain open set.

\begin{definition}
Let $\sigma:V\to X$ be an infinitesimal manifold submetry, and $\mathcal{S}\In \B(\sigma)$ a subset of the algebra $\B(\sigma)$ of basic polynomials. 
\begin{itemize}
\item The set $\mathcal{S}$ is called a \emph{separating set} for $\sigma$ if it separates the fibers of $\sigma$, that is, if $\Le(\mathcal{S})=\F_\sigma$.
\item The set $\mathcal{S}$ is called a \emph{local separating set}  for $\sigma$ if there exists an open subset $U\In V$ with the following properties. (1) $U$ is $\Le(\mathcal{S})$-saturated, hence also $\F_\sigma$-saturated. And (2) $\mathcal{S}$ separates fibers of $\sigma$ contained in $U$, that is,  $\Le(\mathcal{S})|_U=\F_\sigma|_U$.
\end{itemize}
\end{definition}

\begin{example} 
To illustrate the requirement that the open set $U$ be saturated with respect to both $\Le(\mathcal{S})$ and $\F_\sigma$ in the definition above, consider $V=X=\R$, and $\sigma=\operatorname{Id}$. Then the leaves of  $\F_\sigma$ are points, and every polynomial on $V$ is $\sigma$-basic, so $\B(\sigma)=\R[t]$. The set $\mathcal{S}=\{t^2\}$ separates fibers of $\sigma$ contained in $U=(0,\infty)$, an open set which is $\F_\sigma$-saturated but not $\Le(\mathcal{S})$-saturated. In fact, $\mathcal{S}$ is \emph{not} a local separating set for $\sigma$.
\end{example}

In the context of group actions, sets of separating invariants have been intensely studied in the recent past, see \cite[Section 2.4]{DerksenKemper} and references therein. %\todo{ask Harm or Kemper for a list of important works?}

We will need  \cite[Proposition 37]{MR19}, which we state for the reader's convenience:
\begin{lemma}
\label{L:Lapiff}
Let $\mathcal{S}\In\R[V]$ be a subset containing $r^2$. Then the sub-algebra $B\In\R[V]$ generated by $\mathcal{S}$ is Laplacian if and only if $\Delta f$ and $\scal{\nabla f, \nabla g}$ belong to $B$, for every $f,g\in \mathcal{S}$.
\end{lemma}

\begin{remark}
\label{R:quadratic}
Suppose $A$ is an algebra containing $r^2$ and generated by homogeneous polynomials of degree $2$, and let $A_2$ be its degree two graded part. Then, identifying $V=\R^n$ and quadratic polynomials $f\in\R[V]_2$ with the symmetric matrices $\Hess(f)/2\in\Sym^2(\R^n)$, Lemma \ref{L:Lapiff} implies that $A$ is Laplacian if and only if $A_2$ is closed under the standard Jordan product on $\Sym^2(\R^n)$ given by $M\bullet N= (MN+NM)/2$. This observation leads to a classification of Laplacian algebras generated by quadratic polynomials: they are essentially the ones given in Examples \ref{E:FFTorthogonal} and \ref{E:Clifford} below --- see \cite{MR20}.
\end{remark}

\begin{lemma}
\label{L:localglobal}
Let $\F, \F'$ be the decompositions of $V$ into the fibers of infinitesimal manifold submetries $\sigma, \sigma'$ from $V$, and let $U\In V$ be an open subset which is both $\F$- and $\F'$-saturated. If $\F|_U=\F'|_U$, then $\F=\F'$.
\end{lemma}
\begin{proof}
Let $x,y\in V$ on the same $\F$-leaf. Let $L\In U$ an $\F$-, hence also $\F'$-leaf. Take a minimizing geodesic $\gamma:[0,l]\to V$ from $x$ to $L$, with $\gamma(0)=x$, and $\gamma(l)\in L$. Choose a vector $v\in T_yV$ whose image in (the appropriate space of directions in) the quotient $V/\F$ coincides with the image of $\gamma'(0)$.

By \cite[Proposition 14(4)]{MR19}, $\gamma(t)$ and $y+tv$ belong to the same $\F$-leaf for all $t$. In particular, there is $\epsilon>0$ such that, for all $t\in(l-\epsilon, l+\epsilon)$, the points $\gamma(t)$ and $y+tv$ belong to the same $\F'$-leaf. Thus $\gamma'(l)$ and $v\in T_{y+lv}$ map to the same vector in $V/\F'$.

Applying  \cite[Proposition 14(4)]{MR19} again, we conclude that $x=\gamma(0)$ and $y$ belong to the same $\F'$-leaf.

Thus $\F'$ is coarser than $\F$. Reversing the roles of $\F,\F'$ in the argument above yields $\F$ coarser than $\F'$, therefore $\F=\F'$.
\end{proof}

\begin{theorem}
\label{T:separating}
Let $\sigma$ be an infinitesimal manifold submetry from $V$, and $\mathcal{S}$ a local separating set for $\sigma$ containing $r^2$. Then the sub-algebra $B\In\R[V]$ generated by $\mathcal{S}$ coincides with $\B(\sigma)$ if and only if $\Delta f$ and $\scal{\nabla f, \nabla g}$ belong to $B$, for every $f,g\in \mathcal{S}$.
\end{theorem}
\begin{proof}
Suppose $B=\B(\sigma)$. Then, by \cite[Theorem A]{MR19}, $B$ is Laplacian, and therefore $\Delta f$ and \[\scal{\nabla f, \nabla g}=\frac{\Delta(fg) -f\Delta g -g\Delta f}{2}\]
 belong to $B$, for every $f,g\in \mathcal{S}$.

For the converse implication, assume $\Delta f$ and $\scal{\nabla f, \nabla g}$ belong to $B$, for every $f,g\in \mathcal{S}$. By Lemma \ref{L:Lapiff}, $B$ is a Laplacian algebra. By  Corollary \ref{MC:correspondence}, $B=\B(\sigma')$ for some manifold submetry $\sigma'$. 

Denoting by $\F, \F'$ the fiber decompositions associated to $\sigma, \sigma'$, the fact that $\mathcal{S}\In B\In \B(\sigma)$ implies that $\Le(\mathcal{S})<\F'<\F$ (recall that ``$<$'' means ``coarser than'').

By definition of local separating, there exists an $\Le(\mathcal{S})$-saturated open subset of $V$ (hence both $\F$-, and $\F'$-saturated) such that $\Le(\mathcal{S})|_U=\F|_U$. Thus $\F'|_U=\F|_U$, which, by Lemma \ref{L:localglobal}, implies that $\F'=\F$. In particular, $B=\B(\F')=\B(\F)=\B(\sigma)$.
\end{proof}

\todo{removed proof of Corollary \ref{MC:separating}}
%\begin{proof}[Proof of Corollary \ref{MC:separating}]
%Follows immediately from Theorem \ref{T:separating}, where the infinitesimal manifold submetry $\sigma$ is taken to be the quotient map $V\to V/G$.
%\end{proof}

\subsection{Applications to First Fundamental Theorems}
\label{SS:FFT}

We collect here a few simple examples which illustrate how Corollary \ref{MC:separating} and Theorem \ref{T:separating} can be used to prove First Fundamental Theorems, that is, to prove that certain sets of invariant (respectively basic) polynomials actually generate the algebra of all invariant (respectively basic) polynomials. Our method is loosely analogous to the method illustrated in \cite[Section 5]{Schrijver08} to prove FFT's for general tensors (as opposed to symmetric tensors, that is, polynomials).  %\todo{A more involved example will be given in the Appendix?} 

\begin{example}
\label{E:FFTorthogonal}
The standard diagonal action of $G=\OO(n)$ (respectively, $\U(n)$, $\Sp(n)$) on $V= (\R^n)^k$ (respectively $(\C^n)^k$, $(\HH^n)^k$). Generators for the algebra of invariants $\R[V]^G$ are given by the polynomials $f_{ij}(v_1, \ldots, v_k)=\langle v_i, v_j\rangle$, and similarly in the complex and quaternionic cases. This is sometimes called Weyl's First Fundamental Theorems for $\OO(n)$ (respectively $\U(n)$, $\Sp(n)$), see \cite{Weyl}.

Indeed, these polynomials are clearly $G$-invariant, and it is an elementary fact in Linear Algebra that they separate the $G$-orbits. Moreover, being quadratic polynomials, it follows that $\Delta(f_{ij})$ are constant, hence  belong to the algebra generated by $f_{ij}$. Finally, a simple computation shows that $\langle \nabla f_{ab}, \nabla f_{cd}\rangle$ belongs to the span of $\{f_{ij}\}$ for every $a,b,c,d$.
\end{example}

\begin{example}
\label{E:FFTSO}
The standard diagonal action of $G=\SO(n)$ on $V= (\R^n)^k$. A set of generators for $\R[V]^G$ is given by the $f_{ij}$ from the previous example when $k<n$. If $k\geq n$, one needs to add certain polynomials $h_S$, one for each $n$-element subset $S=\{s_1, \ldots, s_n\}\In \{1, \ldots, k\}$, defined by $h_S(v_1, \ldots v_k)=\det(v_{s_1}, \ldots v_{s_n})$.

Indeed, it is easy to see that these polynomials are $G$-invariant, and separate $G$-orbits. Moreover, $\Delta h_S=0$ because each variable only appears once, and $\langle \nabla h_S, \nabla h_T \rangle$, being $\OO(n)$-invariant, can be written as polynomials in the $f_{ij}$, by the previous example. Finally, the fact that the gradient of the determinant function is the adjugate matrix, together with Cramer's rule, shows that $\langle \nabla f_{ij}, \nabla h_S\rangle$ vanishes if ${|S\cap \{i,j\}|=0}$ or $2$, and equals $\pm h_{j\cup S\setminus i}$ if $i\in S$ and $j\not\in S$.
\end{example}

\begin{example}\label{E:Clifford}
Clifford foliations. Let $P_0, \ldots P_m$ be a Clifford system, that is, a set of $2l\times 2l$ real symmetric matrices such that $P_i^2=I$ for every $i$, and $P_i P_j=-P_j P_i$ for every $i\neq j$. The associated Clifford foliation in $\R^{2l}$ is defined as $\Le(\{r^2, f_0, \ldots, f_m\})$, where $f_i(x)=\langle P_i x, x \rangle$, see \cite{Radeschi14}. Then the algebra of basic polynomials is generated by the quadratic polynomials $r^2, f_0, \ldots f_m$.  (This was originally proved in \cite{MR20}, with a considerably more complicated argument.)

Indeed, they separate leaves by definition, their Laplacians are constant, and
$ \langle \nabla f_i, \nabla f_j \rangle(x)=\langle (P_i P_j + P_j P_i)x, x \rangle$, which equals either $0$ or $2\|x\|^2$, according to whether $i\neq j$ or $i=j$. 
\end{example}

\section{Polarizations}\label{S:pols}

Polarizations are a classical tool in Invariant Theory, used to produce multi-variable invariants from single-variable invariants. We introduce a generalization and show that in many situations these ``generalized polarizations'' are enough to generate all multi-variable invariants. For a general reference on (classical) polarizations, see \cite[Section 7.1]{KraftProcesi} and \cite{Weyl}.

\subsection{Definitions}\label{SS:poldefs}

\begin{definition}
Let $\mathcal{S}$ be an arbitrary subset of $\R[V]$. The \emph{Laplacian sub-algebra generated by} $\mathcal{S}$ is defined to be the smallest Laplacian sub-algebra $A$ of $\R[V]$ that contains $\mathcal{S}$. That is, $A$ is the intersection of all Laplacian sub-algebras of $\R[V]$ containing $\mathcal{S}$.
\end{definition}

\begin{remark}\label{R:quadratic2}
In the notation of Remark \ref{R:quadratic}, if the elements of $\mathcal{S}$ are homogeneous polynomials of degree two, then the Laplacian algebra generated by $\mathcal{S}$ is the sub-algebra of $\R[V]$ generated by the smallest Jordan sub-algebra of $\Sym^2(\R^n)\simeq \R[V]_2$ containing $\mathcal{S}$ and $\operatorname{Id}$.
\end{remark}

\begin{remark}
If $\mathcal{S}$ generates $A$ as a Laplacian algebra as in the definition above, one can produce a larger set which generates $A$ as an algebra in the following way. In particular, if $\mathcal{S}$ is a local separating set for an infinitesimal manifold submetry $\sigma$, this produces a generating set for $A=\B(\sigma)$ by Theorem \ref{T:separating}.

Let $\mathcal{S}_1=\mathcal{S}\cup \{r^2\}$.  For $l\geq 2$, define $\mathcal{T}_l$ as all the elements of
\[\Delta \mathcal{S}_{l-1} \cup \{\langle \nabla f, \nabla h\rangle \mid f,h\in \mathcal{S}_{l-1} \} \]
 that do not belong to the algebra generated by $\mathcal{S}_{l-1}$, and define $\mathcal{S}_l=\mathcal{S}_{l-1}\cup\mathcal{T}_l$.

Then $\cup_{l=1}^\infty \mathcal{S}_{l}$ generates $A$ as an algebra by Lemma \ref{L:Lapiff}. Moreover, since $A$ is finitely generated by \cite[Lemma 24]{MR19}, the nested sequence $\mathcal{S}_l$ stabilizes, and $A$ is generated, as an algebra, by $\mathcal{S}_l$ for some $l$ that depends only on $\mathcal{S}$.
\end{remark}

\begin{definition}
Let $A\In\R[V]$ a Laplacian sub-algebra, associated to the infinitesimal manifold submetry $\F$, and $k\geq 1$ an integer. The algebra of \emph{generalized polarizations} on $V^k$, denoted $A^{(k)}$, is the Laplacian sub-algebra of $\R[V^k]$ generated by the polynomials $f_{ij}(v_1, \ldots, v_k)=\langle v_i, v_j\rangle$, and the polynomials $(v_1, \ldots, v_k)\mapsto f(v_1)$, for all $f\in A$.

Denote by $\F^{(k)}$ the infinitesimal manifold submetry of $V^k$ corresponding to $A^{(k)}$ according to  Corollary \ref{MC:correspondence}.
\end{definition}

Note that the subspaces $V_1=V\times \{0\} \times \cdots \times \{0\}$, $V_2=\{0\}\times V\times \{0\} \times \cdots \times \{0\}$, etc are $\F^{(k)}$-saturated (also called $\F^{(k)}$-invariant). In particular, $A^{(k)}$ is multi-graded with respect to the product structure of $V^k$, by \cite[Proposition 19]{MR20}.

Note also that, if $\F$ is homogeneous, given by the orbits of some $G\In\OO(V)$, then $A^{(k)}$ is contained in the algebra of invariants $\R[V^k]^G$, where $G$ acts on $V^k$ diagonally. In particular, $A\mapsto A^{(k)}$ can be seen as a procedure that produces a set of multi-variable invariants from the single-variable invariants.

\begin{example} Classical polarizations. Fix an orthonormal basis $x^1, \ldots, x^n$ of $V^*$, and obtain from it an orthonormal basis \[\{x_i^a\ | \ i=1,\ldots, k,\ a=1,\ldots, n\}\] for $(V^k)^*$. Then, for any $1\leq i,j\leq k$, the \emph{classical polarization differential operator} $P_{ij}$ on $\R[V^k]$, defined by
\[ P_{ij}=\sum_{a=1}^n x_i^a \frac{\partial}{\partial x_j^a}, \]
preserves $A^{(k)}$. Indeed, for any $H=H(x_1, \ldots x_k)\in A^{(k)}$, which we may assume to be (multi-)homogeneous, we have
\[ \langle \nabla f_{ij}, \nabla H\rangle_{V^k} = P_{ij} H + P_{ji} H \in A^{(k)}\]
because  $A^{(k)}$ is Laplacian. Since $P_{ij} H$ and $P_{ji} H$ are homogeneous of different multi-degrees (unless $i=j$, in which case they are equal), and $A^{(k)}$ is multi-graded, they both belong to $A^{(k)}$. 

Given $f\in A$, homogeneous of degree $d$, define $F(x_1, \ldots, x_n)=f(x_1)$, which is in $ A^{(k)}$ by definition. The polynomials obtained from such $F$ by repeated application of classical polarization operators 
%\[P_{i1}(H), P_{i1}^2(H), \ldots, P_{i1}^d(H) \in A^{(k)}\]
are called \emph{classical polarizations}, and the algebra $\operatorname{pol}(A,k)\subset A^{(k)}$ generated by all classical polarizations of all $f\in A$ is called the algebra of classical polarizations. Note that since the classical polarization operators are of order one, they satisfy the Leibniz rule, and so we would obtain the same algebra $\operatorname{pol}(A,k)$ if we only let $f$ run through a set of generators of $A$. Moreover,  $\operatorname{pol}(A,k)$ equals the smallest sub-algebra of $\R[V]$ that contains all polynomials $F$ of the form above, and is preserved by the polarization operators.

In particular, note that $(P_{i1})^d(F)(x_1, \ldots, x_k)=d! f(x_i)$ is a classical polarization. Thus, the restriction $A^{(k)}|_{V_i}$ of $A^{(k)}$ to each invariant subspace $V_i\subset V^k$ contains a copy of $A$. If $\F$ is homogeneous, then $A^{(k)}|_{V_i}=A$. In contrast, for inhomogeneous $\F$ and $k\geq 2$, the authors do not know of a single example where these two algebras coincide, see Subsection \ref{SS:homogeneity} below.
\end{example}

\begin{example} \label{E:Wallach} Wallach's polarizations \cite[Appendix 2]{Wallach93}. Given $i,j$, and $f(x)\in A$ homogeneous, let $F(x_1, \ldots, x_k)=f(x_i)$, and define the \emph{Wallach polarization operator}, denoted $P^f_{ij}$, by
\[ P^f_{ij}=\sum_{a=1}^n \frac{\partial F}{\partial x_i^a} \frac{\partial}{\partial x_j^a}.\]
Note that for $f=r^2/2$, one recovers the classical operator $P_{ij}$. 

The operators $P^f_{ij}$ preserve $A^{(k)}$. Indeed, let $Q=P_{ji}(F)\in A^{(k)}$. Then, for any multi-homogeneous $H\in A^{(k)}$, we have
\[A^{(k)}\ni \langle \nabla Q, \nabla H\rangle_{V^k} = P^f_{ij}(H) + \sum_{a,b} x_j^a \frac{\partial^2 F}{\partial x_i^a\partial x_i^b}\frac{\partial H}{\partial x_i^b}.\]
If $i=j$, the two summands above coincide, while for $i\neq j$, they are homogeneous with different multi-degrees. In either case, $P^f_{ij}(H)\in A^{(k)}$.

In particular, if we put $H(x_1,\ldots, x_k)=h(x_j)$ for some $h\in A$, we obtain the generalized polarization $P^f_{ij}(H)$, given by 
\[(x_1,\ldots, x_k)\mapsto \langle (\nabla f)(x_i),  (\nabla h)(x_j) \rangle_V.\]

\end{example}

\subsection{Homogeneity of generalized polarizations}
The goal of this subsection is to prove Theorem \ref{MT:hompol}.

%The technical condition alluded to above is the following:
\begin{definition}[$k$ normal spaces]
\label{D:k-NS}
Let $(V,\F)$ be an infinitesimal manifold submetry. We say $\F$ satisfies the \emph{$k$ normal spaces} condition, abbreviated $k$-NS, if there exists a non-empty open subset $U\In V^k$ such that, for all $(x_1, \ldots, x_k)\in U$, we have
\[ \nu_{x_1}(L_{x_1}) + \cdots + \nu_{x_k}(L_{x_k}) =V, \] 
where for any $x\in V$, $\nu_{x}(L_{x})\In V$ denotes the normal space to the $\F$-leaf $L_x$ at $x$. 

If $G\subset \OO(V)$ is a compact subgroup, we say $G$ satisfies $k$-NS if its orbit decomposition does.
%\todo{is $k$-NS in some sense generic?} who cares!
\end{definition}

\begin{remark} 
\label{R:k-NS}The following are immediate consequences of the definition.
If $\F$ satisfies $k$-NS, then it also satisfies $l$-NS for all $l\geq k$. Denoting by $\F^0$  the decomposition of $V$ into the connected components of the $\F$-leaves,  $\F^0$ satisfies $k$-NS if and only if $\F$ does. If $\F<\F'$ (coarser than) and $\F$ satisfies $k$-NS, then so does $\F'$.
\end{remark}

Recall that the \emph{principal stratum} of an infinitesimal manifold submetry $(V,
F)$ is the subset $V_0\subset V$ of all points $x\in V$ such that $\nu_x(L_x)$ is spanned by $\{ \nabla f(x) \mid f\in \B(\F)\}$. It is non-empty and Zariski-open, in particular it is dense in the Euclidean topology.

The next lemma implies that $k$-NS is equivalent to the condition that $k$ \emph{generic} normal spaces span $V$, and will ensure that Theorem \ref{T:separating} may be applied in the proof of Theorem \ref{MT:hompol} below.

\begin{lemma}
\label{L:k-NS}
Let $(V,\F)$ be an infinitesimal manifold submetry satisfying the $k$-NS condition. Then the open subset $U$ in Definition \ref{D:k-NS} may be assumed to be dense, $\F^{(k)}$-saturated, and contained in $(V_0)^k$, where $V_0$ denotes the principal stratum of $\F$.
\end{lemma}
\begin{proof}
Let $\rho_1, \ldots \rho_r \in A$ be a generating set. Since $\nabla f (x)\in \nu_{x}(L_{x})$ for every $x\in V$ and $f\in A$, the set
\[ U'= \left\{ (x_1, \ldots, x_k)\in (V_0)^k\ |\ \operatorname{span}\{\nabla\rho_a(x_b)\ |\ 1\leq a\leq r,\ 1\leq b\leq k\}=V \right\}\]
satisfies the condition in Definition \ref{D:k-NS}. Moreover, it is the complement of a Zariski-closed set, defined by polynomial equations in the polynomials $\langle (\nabla \rho_a)(x_i),  (\nabla \rho_b)(x_j) \rangle_V$. Since these are generalized polarizations by Example \ref{E:Wallach}, we conclude that $U'$ is $\F^{(k)}$-saturated. Finally, $U'$ contains $U\cap (V_0)^k$. In particular, the Zariski-open set $U'$ is non-empty, therefore dense.
\end{proof}

 \todo{moved old Thm (now Theorem \ref{MT:hompol}) to the Intro now}
%\begin{theorem}[Homogeneity of generalized polarizations]
%\label{T:hompol}
%Let $(V,\F)$ be an infinitesimal manifold submetry satisfying $k$-NS, with associated algebra of basic polynomials $A=\B(\F)$. Let $\OO(\F)$ be the closed subgroup of $\OO(V)$ consisting of all $g\in\OO(V)$ that map each $\F$-leaf to itself, and consider the diagonal action of $\OO(\F)$ on $V^k$. Then $A^{(k)}=\R[V^k]^{\OO(\F)}$.
%\end{theorem}
\begin{proof}[Proof of Theorem \ref{MT:hompol}]
Let $\F=\F_\sigma$ be the decomposition of $V$ into $\sigma$-fibers, and $G=\OO(\sigma)=\OO(\F)$. Since  $A\In\R[V]^G$, we have $A^{(k)}\In\R[V^k]^G$. With Theorem \ref{T:separating} in mind, in order to prove the equality $A^{(k)}=\R[V^k]^G$, it suffices to show that $A^{(k)}$ is a local separating set for the decomposition of $V^k$ into the $G$-orbits, because $A^{(k)}$ is Laplacian by definition.

Let $U\In V^k$ be an open subset as in Definition \ref{D:k-NS}, which, by Lemma \ref{L:k-NS}, we may assume to be $\F^{(k)}$-saturated and contained in $(V_0)^k$. We will show that $A^{(k)}$ separates $G$-orbits contained in $U$. Explicitly, given 
\[(x_1, \ldots, x_k), (y_1, \ldots, y_k)\in U,\]
 such that every generalized polarization $H\in A^{(k)}$ takes the same value on them, that is, $H(x_1, \ldots, x_k)=H(y_1, \ldots, y_k)$, we will construct an orthogonal transformation $g\in \OO(V)$ such that $g(x_1, \ldots, x_k)=(y_1, \ldots, y_k)$, and then prove that $g\in G$.

Let $\rho_1, \ldots \rho_r \in A$ be a generating set. We may assume $\rho_1=r^2/2$. Consider the following two $kr$-tuples of vectors in $V$:
\[ (\nabla\rho_a(x_i))_{a,i} \qquad  (\nabla\rho_a(y_i))_{a,i}\]
The inner product of any two entries in the first tuple coincides with the inner product of the corresponding entries in the second tuple, because such inner products are generalized polarizations, see Example \ref{E:Wallach}. By Weyl's First Fundamental Theorem for the orthogonal group (see Example \ref{E:FFTorthogonal}), there exists $g\in\OO(V)$ such that 
$g\nabla\rho_a(x_i) =\nabla\rho_a(y_i)$ for all $a,i$. Since $\rho_1=r^2/2$, this implies that $gx_i=y_i$ for all $i$. Incidentally, $g$ is uniquely determined because $\{\nabla\rho_a(x_i)\}_{a,i}$ span $V$.

It remains to show that $g\in G$, that is, that $g$ takes each $\F$-leaf to itself. Let $v\in V$ be arbitrary. We will show that $x_1+v$ and $g(x_1+v)=y_1+gv$ belong to the same $\F$-leaf. By the $k$-NS assumption, there exist scalars $\lambda_{ai}$ such that 
\[v=\sum_{a,i} \lambda_{ai}\nabla\rho_a(x_i).\]

We will need the following Claim, which follows immediately from the definition of the Wallach polarization operators (see Example \ref{E:Wallach}) :\\
{\bf Claim:} 
Let $H\in A^{(k)}$ be an arbitrary generalized polarization. Consider the function $\Omega:\R\times V^k\to \R$ given by
\[\Omega(t, z_1, \ldots, z_k)=H\left(z_1 + t\sum_{a,i} \lambda_{ai}\nabla\rho_a(z_i),\ \ z_2, \ \ \ldots\ \ ,\ \  z_k \right)\]
Then 
\[\frac{\partial\Omega}{\partial t}(t, z_1, \ldots, z_k)=\tilde{H}\left(z_1 + t\sum_{a,i} \lambda_{ai}\nabla\rho_a(z_i),\ \  z_2, \ \ \ldots\ \ ,\ \ z_k \right)\]
where $\tilde{H}=\sum_{a,i} \lambda_{ai}P^{\rho_a}_{i1}(H)\in A^{(k)}$. In particular, $\frac{\partial^b\Omega}{\partial t^b}|_{t=0} \in A^{(k)}$ 
for all $b$.

We apply the Claim above to the case where $H(z_1, \ldots, z_k)=h(z_1)$ for an arbitrary $h\in A$. It follows that the polynomials $t\mapsto h(x_1+tv)$ and  $t\mapsto h(g(x_1+tv))$ are equal, because they have the same derivatives, of all orders, at $t=0$. Since $h\in A$ is arbitrary, this shows that $x_1+v$ and $g(x_1+v)$ belong to the same orbit, and, since $v$ was arbitrary, we conclude  that $g\in G$.
\end{proof}

We finish this subsection with a couple of open questions regarding the $k$-NS condition.
\begin{question}
When $k\geq 2$, can the hypothesis that $\F$ satisfy the $k$-NS condition be dropped from Theorem \ref{MT:hompol}?
\end{question}

\begin{question}
Is the $k$-NS condition equivalent to $k\dim(V/\F)\geq \dim(V)$?
\end{question}

\subsection{When generalized polarizations generate}
Let $G\In \OO(V)$ be a closed subgroup, with algebra of invariants $A=\R[V]^G$, $k$ a positive integer, and let $A^{(k)}$ be the algebra of generalized polarizations. As mentioned earlier, we have $A^{(k)}\subset \R[V^k]^G$.  The objective of this subsection is to give sufficient conditions for equality to hold, and in particular to prove Theorem \ref{MT:polarization2}. 

As mentioned in the Introduction, equality cannot hold in full generality, because generalized polarizations depend only on the algebra of invariants $A=\R[V]^G$, which in turn only depends on the $G$-orbits, and not on $G$ itself. On the other hand, for $k=\dim(V)$ (and hence for all $k\geq\dim(V)$), the group $G$ itself is a $G$-orbit, in particular it is determined by $\R[V^k]^G$. Indeed, $V^{\dim(V)}$ can be identified with $\operatorname{End}(V)$, and $G$ is the $G$-orbit through $\operatorname{Id}\in \operatorname{End}(V)$. For a concrete example, compare Examples \ref{E:FFTorthogonal} and \ref{E:FFTSO}: the groups $\OO(n)$ and $\SO(n)$ have the same orbits, hence have the same algebra of invariants and algebras of generalized polarizations. But the algebras $\R[V^k]^{\OO(n)}$ and $\R[V^k]^{\SO(n)}$ are distinct when $k\geq n$.

Subgroups $G,G'\subset \OO(V)$ are called \emph{orbit-equivalent} if they have the same orbits. The discussion above shows that, at least when $k$ is large, we can only expect $A^{(k)}= \R[V^k]^G$ if $G$ is \emph{maximal} (with respect to inclusion) in its orbit-equivalence class.

%\begin{theorem}
%\label{T:polgen}
%Let $G\In \OO(V)$ be a closed subgroup, with algebra of invariants $A=\R[V]^G$. Assume that $G$ is maximal in its orbit-equivalence class, and that the decomposition of $V$ into $G$-orbits satisfies $k$-NS. Then $\R[V^k]^G=A^{(k)}$.
%\end{theorem}
\todo{removed a theorem from here, incorporated statement into Corollary \ref{MC:polarization2}}

\begin{proof}[Proof of Theorem \ref{MT:polarization2}]
Being ``maximal in its orbit-equivalence class'' is just another way of saying that $G=\OO(\sigma)$, where $\sigma$ is the quotient map $V\to V/G$. Thus the first part of the statement follows immediately from Theorem \ref{MT:hompol}.
\begin{enumerate}[a)]
\item  By Remark \ref{R:k-NS}, it suffices to take $G$ to be a maximal torus in $\OO(V)$, and show that $G$ has the $2$-NS property. If $V$ is even-dimensional, we may identify $V$ with $\C^n$, and let $G=\U(1)^n$ act on $V$ in the standard way. Then the normal spaces at a pair points near $(1, \ldots, 1)$ and $(\sqrt{-1}, \ldots, \sqrt{-1})$ span $V$, so $G$ has $2$-NS. If $V$ is odd-dimensional, we have $V=\C^n\oplus\R$ with $G=\U(1)^n$ acting only on the $\C^n$ factor. Again, the normal spaces at a pair points near $(1, \ldots, 1, 0)$ and $(\sqrt{-1}, \ldots, \sqrt{-1}, 0)$ span $V$, so $G$ has $2$-NS.
\item $k\geq \dim(V)$ implies $k$-NS because the position vector $v$ is normal to the $G$-orbit through $v$, for any $v\in V$.
\end{enumerate}
\end{proof}

\begin{remark} 
\label{R:altproof}
Theorem  \ref{MT:polarization2} gives an alternative proof of Corollary \ref{MC:polarization}, which avoids \cite[Theorem 3.4]{DKW08}, because $G$ finite implies both $k$-NS for all $k$, and maximality in its orbit-equivalence class (see \cite[Lemma 1]{Swartz02}).
\end{remark}

\begin{remark}
If $G\subset\OO(V)$ is maximal in its orbit-equivalence class, the algebra $A=\R[V]^G$ determines $G$, hence it also determines $\R[V^k]^G$ for all $k$. As a corollary of Theorem \ref{MT:polarization2}, we obtain an explicit way to produce $\R[V^k]^G$ out of $A=\R[V]^G$. Namely, out of $A$ one constructs $A^{(\dim(V))}\subset \R[V^{\dim(V)}]$, which equals $\R[V^{\dim(V)}]^G$ by Theorem \ref{MT:polarization2}, and therefore  $\R[V^k]^G$ is the restriction of $A^{(\dim(V))}$ to the subspace $V^k\subset V^{\dim(V)}$.
\end{remark}

\begin{remark}
\label{R:polproperty}
Theorem \ref{MT:polarization2} applies to many groups $G$. In contrast, $\R[V^k]^G$ being generated by classical polarizations  is quite special.

Schwarz \cite{Schwarz07} has classified the complex rational representations $V$ of simple reductive algebraic groups $G$ such that $\C[V^k]^G$ is generated by classical polarizations. In particular the irreducible ones are coregular, that is, $\C[V]^G$ is free, and the reducible ones are isomorphic to a direct sum of a certain number of copies of the standard action of $\operatorname{SL(n)}$ on $\C^n$. It is also conjectured in \cite{Schwarz07} that if $G$ is a finite group and $\C[V^k]^G$ is generated by classical polarizations, then the action of $G$ on $V$ is generated by reflections.

Even if $G$ is generated by reflections, $\C[V^2]^G$ (and hence $\C[V^k]^G$ for $k\geq 3$) may fail to be generated by classical polarizations, for example when $G$ is the Weyl group of type $D_4$. As for Wallach polarizations, they \emph{are} enough to generate $\C[V^2]^G$ for $G$ of type $D_n$ for all $n$, but \emph{not} for $G$ of type $F_4$. See \cite[Appendix 2]{Wallach93} and \cite{Hunziker97}.
\end{remark}

\subsection{Homogeneity of infinitesimal manifold submetries}
\label{SS:homogeneity}

Not all infinitesimal manifold submetries $(V,\F)$ are homogeneous, that is, given by the orbit decomposition of some compact subgroup $G\subset \OO(V)$. Historically, the first inhomogeneous examples were isoparametric foliations constructed in \cite{OT75, OT76}, later generalized in \cite{FKM81}, and the octonionic Hopf fibration of $\R^{16}$. The latter is an example of a Clifford foliation, see Example \ref{E:Clifford}. All these are examples of \emph{composed Clifford foliations}, see \cite{Radeschi14}. In some sense ``most''  composed Clifford foliations are inhomogeneous \cite{GR16}, and, since this construction can be applied to any homogeneous foliation, one can reasonably argue that there are ``at least as many'' inhomogeneous manifold submetries as homogeneous ones.

 There are many ways one can prove that a given manifold submetry $(V,\F)$ is inhomogeneous. One method involves generalized polarizations. Recall from Subsection \ref{SS:poldefs} that the restriction of $\F^{(k)}$ to the invariant subspace $V\times 0 \times \cdots \times 0\subset V^k$ is finer than the original manifold submetry $\F$, and that they coincide if $\F$ is homogeneous. Thus, if one can find a number $k$ and a generalized polarization in $A^{(k)}$ whose restriction to $V\times 0 \times \cdots \times 0$ is not constant on some $\F$-leaf, then $\F$ is inhomogeneous. The following result provides a converse to this method, under the technical $k$-NS assumption:
\begin{theorem}
\label{T:homogeneity}
Let $(V,\F)$ be an infinitesimal manifold submetry with algebra of basic polynomials $A$, and $k\geq 2$. Assume $(V,\F)$ satisfies $k$-NS, and that the restriction of $\F^{(k)}$ to the saturated subspace $V\times \{0\}\times\cdots\times\{0\}$ is isomorphic to $\F$. Then $\F$ is homogeneous.
\end{theorem}
\begin{proof}
By Theorem \ref{MT:hompol}, $\F^{(k)}$ is given by the orbits of the diagonal action of $\OO(\F)$ on $V^k$. In particular, its restriction to $V\times \{0\}\times\cdots\times\{0\}$ is given by the orbits of $\OO(\F)$ acting on $V$. Thus $\F$ is also given by these orbits, that is, $\F$ is homogeneous.
\end{proof}

As a corollary, if the restriction of $\F^{(\dim(V))}$ to the saturated subspace $V\times \{0\}\times\cdots\times\{0\}$ is isomorphic to $\F$, then $\F$ is homogeneous. Note also that the condition in Theorem  \ref{T:homogeneity} can be rephrased as an algebraic condition on $A$, namely ``the restriction of $A^{(k)}$ to  $V\times \{0\}\times\cdots\times\{0\}$ is equal to $A$''.

\begin{remark}\label{R:quadratic3}
In the notations of Remarks \ref{R:quadratic} and \ref{R:quadratic2}, let $A$ be a Laplacian algebra generated by the homogeneous quadratic polynomials $A_2$, seen as a Jordan subalgebra $A_2\simeq J\subset\Sym^2(R^n)$. Let $\mathcal{U}\subset \operatorname{Mat}_{n\times n}(\R)$ be the enveloping algebra of $J$, that is, the span of all products of matrices in $J$. Then it is not hard to see that, for all $k\geq 2$,  the algebra $A^{(k)}$ of generalized polarizations is the algebra generated by those quadratic polynomials in $\R[V^k]_2$ whose Hessians have the form 
\[
\begin{bmatrix}
C_{11} & C_{12} & \cdots &C_{1k} \\
C_{12}^T  & C_{22} & \cdots &C_{2k} \\
\vdots & & &\vdots \\
C_{1k}^T  & \cdots & \cdots  & C_{kk}
\end{bmatrix}
\]
where $C_{ij}\in\mathcal{U}$ for all $i,j$. In particular, the restriction of $A^{(k)}$ to  $V\times \{0\}\times\cdots\times\{0\}$ is independent of $k$, for $k\geq 2$. Thus Theorem \ref{T:homogeneity} implies that $A$ is homogeneous if and only if the restriction of $A^{(2)}$ to  $V\times \{0\}$ is equal to $A$ if and only if every symmetric matrix in $\mathcal{U}$ belongs to $J$.
\end{remark}

Since most Clifford foliations do not satisfy the $2$-NS condition, Remark \ref{R:quadratic3} makes one wonder:
\begin{question}
Does Theorem \ref{T:homogeneity} hold without the $k$-NS condition?
\end{question}

\bibliography{ref}

\begin{thebibliography}{DKW08}

\bibitem[DK15]{DerksenKemper}
Harm Derksen and Gregor Kemper.
\newblock {\em Computational invariant theory}, volume 130 of {\em
  Encyclopaedia of Mathematical Sciences}.
\newblock Springer, Heidelberg, enlarged edition, 2015.
\newblock With two appendices by Vladimir L. Popov, and an addendum by Norbert
  A'Campo and Popov, Invariant Theory and Algebraic Transformation Groups,
  VIII.

\bibitem[DKW08]{DKW08}
Jan Draisma, Gregor Kemper, and David Wehlau.
\newblock Polarization of separating invariants.
\newblock {\em Canad. J. Math.}, 60(3):556--571, 2008.

\bibitem[FKM81]{FKM81}
Dirk Ferus, Hermann Karcher, and Hans~Friedrich M\"{u}nzner.
\newblock Cliffordalgebren und neue isoparametrische {H}yperfl\"{a}chen.
\newblock {\em Math. Z.}, 177(4):479--502, 1981.

\bibitem[GR16]{GR16}
Claudio Gorodski and Marco Radeschi.
\newblock On homogeneous composed {C}lifford foliations.
\newblock {\em M\"{u}nster J. Math.}, 9(1):35--50, 2016.

\bibitem[Hel00]{HelgasonGGA}
Sigurdur Helgason.
\newblock {\em Groups and geometric analysis}, volume~83 of {\em Mathematical
  Surveys and Monographs}.
\newblock American Mathematical Society, Providence, RI, 2000.
\newblock Integral geometry, invariant differential operators, and spherical
  functions, Corrected reprint of the 1984 original.

\bibitem[HT92]{HoweTan}
Roger Howe and Eng-Chye Tan.
\newblock {\em Nonabelian harmonic analysis}.
\newblock Universitext. Springer-Verlag, New York, 1992.
\newblock Applications of ${{\rm{S}}L}(2,{{\bf{R}}})$.

\bibitem[Hun97]{Hunziker97}
M.~Hunziker.
\newblock Classical invariant theory for finite reflection groups.
\newblock {\em Transform. Groups}, 2(2):147--163, 1997.

\bibitem[Ilt98]{Iltyakov98}
A.~V. Iltyakov.
\newblock Laplace operator and polynomial invariants.
\newblock {\em J. Algebra}, 207(1):256--271, 1998.

\bibitem[KL20]{KL20}
Vitali {Kapovitch} and Alexander {Lytchak}.
\newblock {Structure of Submetries}.
\newblock {\em arXiv e-prints}, page arXiv:2007.01325, July 2020.

\bibitem[KP96]{KraftProcesi}
Hanspeter Kraft and Claudio Procesi.
\newblock {\em Classical Invariant Theory: A Primer.}
\newblock 1996.
\newblock Lecture notes available on https://kraftadmin.wixsite.com/hpkraft.

\bibitem[LR18]{LR18}
Alexander Lytchak and Marco Radeschi.
\newblock Algebraic nature of singular {R}iemannian foliations in spheres.
\newblock {\em J. Reine Angew. Math.}, 744:265--273, 2018.

\bibitem[MR20a]{MR20}
R.~A.~E. Mendes and M.~Radeschi.
\newblock Singular {R}iemannian foliations and their quadratic basic
  polynomials.
\newblock {\em Transform. Groups}, 25(1):251--277, 2020.

\bibitem[MR20b]{MR19}
Ricardo A.~E. Mendes and Marco Radeschi.
\newblock Laplacian algebras, manifold submetries and the inverse invariant
  theory problem.
\newblock {\em Geom. Funct. Anal.}, 30(2):536--573, 2020.

\bibitem[M{\"{u}}n80]{Muenzner80}
Hans~Friedrich M{\"{u}}nzner.
\newblock Isoparametrische {H}yperfl\"{a}chen in {S}ph\"{a}ren.
\newblock {\em Math. Ann.}, 251(1):57--71, 1980.

\bibitem[M{\"{u}}n81]{Muenzner81}
Hans~Friedrich M{\"{u}}nzner.
\newblock Isoparametrische {H}yperfl\"{a}chen in {S}ph\"{a}ren. {II}. \"{U}ber
  die {Z}erlegung der {S}ph\"{a}re in {B}allb\"{u}ndel.
\newblock {\em Math. Ann.}, 256(2):215--232, 1981.

\bibitem[NS02]{NeuselSmith}
Mara~D. Neusel and Larry Smith.
\newblock {\em Invariant theory of finite groups}, volume~94 of {\em
  Mathematical Surveys and Monographs}.
\newblock American Mathematical Society, Providence, RI, 2002.

\bibitem[OT75]{OT75}
Hideki Ozeki and Masaru Takeuchi.
\newblock On some types of isoparametric hypersurfaces in spheres. {I}.
\newblock {\em Tohoku Math. J. (2)}, 27(4):515--559, 1975.

\bibitem[OT76]{OT76}
Hideki Ozeki and Masaru Takeuchi.
\newblock On some types of isoparametric hypersurfaces in spheres. {II}.
\newblock {\em Tohoku Math. J. (2)}, 28(1):7--55, 1976.

\bibitem[Rad14]{Radeschi14}
Marco Radeschi.
\newblock Clifford algebras and new singular {R}iemannian foliations in
  spheres.
\newblock {\em Geom. Funct. Anal.}, 24(5):1660--1682, 2014.

\bibitem[Rud76]{Rudin}
Walter Rudin.
\newblock {\em Principles of mathematical analysis}.
\newblock McGraw-Hill Book Co., New York-Auckland-D\"{u}sseldorf, third
  edition, 1976.
\newblock International Series in Pure and Applied Mathematics.

\bibitem[Sch07]{Schwarz07}
Gerald~W. Schwarz.
\newblock When polarizations generate.
\newblock {\em Transform. Groups}, 12(4):761--767, 2007.

\bibitem[Sch08]{Schrijver08}
Alexander Schrijver.
\newblock Tensor subalgebras and first fundamental theorems in invariant
  theory.
\newblock {\em J. Algebra}, 319(3):1305--1319, 2008.

\bibitem[Swa02]{Swartz02}
Ed~Swartz.
\newblock Matroids and quotients of spheres.
\newblock {\em Math. Z.}, 241(2):247--269, 2002.

\bibitem[Wal93]{Wallach93}
Nolan~R. Wallach.
\newblock Invariant differential operators on a reductive {L}ie algebra and
  {W}eyl group representations.
\newblock {\em J. Amer. Math. Soc.}, 6(4):779--816, 1993.

\bibitem[Wey97]{Weyl}
Hermann Weyl.
\newblock {\em The classical groups}.
\newblock Princeton Landmarks in Mathematics. Princeton University Press,
  Princeton, NJ, 1997.
\newblock Their invariants and representations, Fifteenth printing, Princeton
  Paperbacks.

\end{thebibliography}
\bibliographystyle{alpha}
\end{document}